\documentclass[12pt]{amsart}
\usepackage{amscd,amssymb,amsfonts}
\usepackage{graphicx}
\usepackage[mathscr]{eucal}

\newtheorem{thm}{Theorem}[section]
\newtheorem{cor}[thm]{Corollary}
\newtheorem{lemma}[thm]{Lemma}

\theoremstyle{definition}
\newtheorem{defn}[thm]{Definition}

\newtheorem{rem}[thm]{Remark}

\renewcommand{\int}[1]{int\left(#1\right)}
\newcommand{\cl}[1]{cl\left(#1\right)}

\newcommand{\R}{{\mathbb R}}

\newcommand{\C}{{\mathbb C}}

\newcommand{\rtt}{{\R^{3,2}}}

\newcommand{\rto}{\R^{2,1}}
\newcommand{\E}{\mathbb{E}^{2,1}} 

\newcommand{\HP}{H^{2}}

\newcommand{\SOTO}{\operatorname{SO}(2,1)}
\newcommand{\SOoTO}{\operatorname{SO}^0(2,1)}

\newcommand{\POTT}{{\operatorname{PO}(3,2)}}
\newcommand{\SOTT}{{\operatorname{SO}(3,2)}}
\newcommand{\sott}{{\operatorname{SO}^0(3,2)}}

\newcommand{\T}{{\mathcal T}}

\renewcommand{\C}{{\mathcal C}}
\newcommand{\N}{{\mathcal N}}
\renewcommand{\L}{{\mathcal L}}

\newcommand{\cone}{\N^{3,2}}
\newcommand{\econe}[1]{\L(#1)}

\newcommand{\Ein}[1]{\operatorname{Ein}_{#1}}
\newcommand{\Eint}{\Ein{3}}
\newcommand{\Einhat}{\widehat{\Eint}}

\newcommand{\CONF}{{\operatorname{Conf}(\Eint)}}

\newcommand{\NF}[1]{\T(#1)} 

\newcommand{\einp}[1]{\Ein{2}(#1)}

\newcommand{\stem}[1]{\mathsf{Stem}(#1)} 
\newcommand{\wing}[1]{\mathsf{W}(#1)} 

\newcommand{\CP}[1]{\C(#1)}
\newcommand{\CHS}[1]{\mathcal{H}(#1)}
\newcommand{\CS}{\mathcal{S}}

\newcommand{\Quad}[1]{\mathsf{Quad}(#1)} 
\newcommand{\allow}[1]{AP(#1)}  

\newcommand{\data}{D}

\newcommand{\conf}[1]{\overline{#1}^{\mbox{\tiny conf}}}

\newcommand{\xm}[1]{{#1}^-}
\newcommand{\xp}[1]{{#1}^+}
\newcommand{\xpm}[1]{{#1}^\pm}

\newcommand{\vz}{\mathsf{z}}

\newcommand{\vu}{\mathsf{u}}

\newcommand{\vy}{\mathsf{y}}
\newcommand{\vx}{\mathsf{x}}
\newcommand{\vv}{\mathsf{v}}

\newcommand{\vp}{\mathsf{p}}

\newcommand{\vo}{\mathsf{0}}

\newcommand{\ldot}[1]{\langle #1\rangle}

\renewcommand{\gg}{\gamma}
\newcommand{\hh}{\eta}
\renewcommand{\tt}{\tau}
\newcommand{\rr}{\rho}
\newcommand{\mm}{\mu}

\begin{document}

\title{Fundamental domains in the Einstein Universe}

\author[Charette]{Virginie Charette}
  \address{D\'epartement de math\'ematiques\\ Universit\'e de Sherbrooke\\
  Sherbrooke, Quebec, Canada}
  \email{v.charette@usherbrooke.ca}

\author[Francoeur]{Dominik Francoeur}
  \address{D\'epartement de math\'ematiques\\ Universit\'e de Sherbrooke\\
   Sherbrooke, Quebec, Canada}
   \email{dominik.francoeur@usherbrooke.ca}

\author[Lareau-Dussault]{Rosemonde Lareau-Dussault}
\address{Department of mathematics\\ University of Toronto\\
  Toronto, Ontario, Canada}
\email{rosemonde.lareau.dussault@mail.utoronto.ca}

\begin{abstract}
We will discuss fundamental domains for actions of discrete groups on the 3-dimensional Einstein Universe.  These will be bounded by crooked surfaces, which are conformal compactifications of surfaces that arise in the construction of Margulis spacetimes.  We will show that there exist pairwise disjoint crooked surfaces in the 3-dimensional Einstein Universe.  As an application, we can construct explicit examples of groups acting properly on an open subset of that space.
\end{abstract}

\thanks{The authors gratefully acknowledge partial support from the
  Natural Sciences and Engineering Research Council of Canada.
 }

\maketitle

\tableofcontents

\section{Introduction}

This paper describes fundamental polyhedra for actions on the 3-dimensional Einstein Universe, bounded by {\em crooked surfaces}.   The Einstein Universe is a Lorentzian manifold, in fact it is the conformal compactification of Minkowski spacetime.  As such, it enjoys a wealth of geometric properties.  A rich diversity of groups act on the Einstein Universe, displaying a wide range of dynamic behavior.

Our interest in the Einstein Universe arises from proper actions of free groups on Minkowski spacetime.  This story starts with Margulis, who constructed examples of 
proper actions on affine 3-space by free non-abelian groups~\cite{Ma83,Ma87}.  Such groups preserve an inner product of signature (2,1) on the underlying vector space of directions.  Thus a quotient of affine space by such an action is endowed with a Lorentzian metric, motivating us to call it a 
{\em Margulis spacetime\/}.

Next, Drumm constructed fundamental domains for these actions, bounded by {\em crooked planes}.  The geometry of crooked planes has been extensively studied by Drumm, Goldman and others~\cite{D92,DG99,BCDG}.

Crooked surfaces were introduced by Frances to study conformal 
compactifications of Margulis spacetimes~\cite{F03}.  (See also~\cite{BCDGM}.)  In current work, Danciger, Gu\'eritaud and Kassel introduce analogous surfaces in anti de Sitter spacetimes~\cite{DGK13}.  Goldman establishes a relation between their new construction and Frances' crooked surfaces~\cite{Go13}.

Some of the earliest work on proper actions on pseudo-Riemannian manifolds is due to Kulkarni; see for example~\cite{K78,K80}.  Frances has examined the dynamics of proper group actions on the Einstein Universe.  In particular, he proved the existence of {\em Lorentzian Schottky groups}, that is, groups acting on the Einstein Universe with ``ping-pong dynamics''~\cite{F05}.  On the other hand, the crooked surfaces considered in~\cite{F03} intersect in a point.  This led us to ask the question\,: how can we obtain {\em disjoint} crooked surfaces in the Einstein Universe?

We are able to do this by adapting a strategy, described in~\cite{BCDG}, based on Drumm's original approach.  Basically, one can ``pull apart'' crooked planes in affine space sharing only a common vertex.  The crucial observation is to see that this pulling apart can also be done at a point on the conformal compactification of the crooked planes, yielding pairwise disjoint surfaces.  The upshot is that we are able to construct explicit examples of groups displaying ping-pong dynamics on the Einstein Universe.

Section~\ref{sec:einstein} introduces the Einstein Universe of dimension 3, and its causal structure.  Namely, as the conformal compactification of Minkowski spacetime, it inherits a conformally flat Lorentzian structure.  Section~\ref{sec:crookedsurface} discusses crooked surfaces.  We first describe them as conformal compactifications of crooked planes and then, propose a ``coordinate-free'' description of crooked surfaces based on Einstein tori.  We also show how to obtain disjoint crooked planes; in Section~\ref{sec:disjoint}, we apply this strategy to get disjoint crooked surfaces.  Section~\ref{sec:schottky} relates our construction to the Lorentzian Schottky groups introduced by Frances.  Finally, in Section~\ref{sec:negative}, we briefly describe an interesting example of disjoint crooked surfaces, arising from a pair of fundamentally incompatible crooked planes in Minkowski spacetime.

\subsection*{Acknowledgements}  Several people listened to talks about this paper while it was in preparation and offered greatly appreciated suggestions and encouragement.  While it would be difficult to mention everyone, we would particularly like to thank Thierry Barbot, Marc Burger, Todd Drumm, Charles Frances, Bill Goldman, Alessandra Iozzi and Anna Wienhard.  The first author thanks IHES for its
hospitality during the initial writing of this paper.   We also would like to thank the anonymous referee for several helpful suggestions and comments.

\section{Einstein Universe}\label{sec:einstein}
The Einstein Universe, $\Ein{n}$, can be defined as the projectivisation of the 
lightcone of $\R^{n,2}$.  We will write everything for $n=3$, as this is the 
focus of the paper.

Let $\rtt$ denote the vector space $\R^5$ endowed with a symmetric bilinear form
of signature $(3,2)$, denoted $\ldot{,}$.  We can choose a basis relative to which, for $\vx=(x_1,x_2,x_3,x_4,x_5)$
and $\vy=(y_1,y_2,y_3,y_4,y_5)\in\R^5$:
\begin{equation*}
 \ldot{\vx,\vy}=x_1y_1+x_2y_2+x_3y_3-x_4y_4-x_5y_5.
\end{equation*}
Let $\vx^\perp$ denote the orthogonal hyperplane to $\vx$\,:
\begin{equation*}
\vx^\perp=\{ \vy\in\rtt~\mid~ \ldot{\vx,\vy}=0\}
\end{equation*}
and $\cone$, the {\em lightcone} of $\rtt$\,:
\begin{equation*}
 \cone=\{ \vx\in\rtt\setminus \vo~\mid~ \ldot{\vx,\vx}=0\} .
\end{equation*}
(Note that we do not include the zero vector in the lightcone.)  

We obtain the Einstein Universe as the quotient of $\cone$ under the action of the non-zero reals by scaling\,: 
\begin{equation*}
\Eint=\cone/\R^* .
\end{equation*}
Denote by $\pi(\vv)$ the image of $\vv\in\cone$ under this projection.  Alternatively, we will write\,: 
\begin{equation*}
\pi(v_1,v_2,\ldots,v_n)=(v_1:v_2:\ldots:v_n)
\end{equation*}
when it alleviates notation.

We can express the orientable double-cover of the Einstein Universe, $\Einhat$, as a quotient as well, this time by the action of the positive reals\,: 
\begin{equation*}
\Einhat=\cone/\R^+ .
\end{equation*}
Any lift of $\Einhat$ to $\cone$ induces a Lorentzian metric on $\Einhat$ by restricting
$\ldot{,}$ to the image of the lift.  For instance, the intersection with
$\cone$ of the sphere of radius 2, centered at $\vo$, consists of vectors $\vx$
such that\,: 
\begin{equation*}
 x_1^2+x_2^2+x_3^2=1=x_4^2+x_5^2.
\end{equation*}
It projects bijectively to $\Einhat$, endowing it with the Lorentzian product
metric $dg^2-dt^2$, where $dg^2$ is the standard round metric on the
2-sphere $S^2$, and $dt^2$ is the standard metric on the circle $S^1$.

Thus $\Eint$ is conformally equivalent to:
\begin{equation*}
 S^2\times S^1/\!\sim, \mbox{ where }\vx\sim-\vx  .
\end{equation*}
Here $-I$ factors into the product of two antipodal maps.  

Any metric on $\Einhat$ pushes forward to a metric on $\Eint$.  Thus $\Eint$ inherits a conformal class of Lorentzian metrics from the ambient
spacetime
$\rtt$.   The group of conformal automorphisms of $\Eint$ is\,: 
\begin{equation*}
\CONF\cong\POTT\cong\SOTT.
\end{equation*}
As $\SOTT$ acts transitively on $\cone$, the group $\CONF$ acts transitively on $\Eint$.

Slightly abusing notation, we will also denote by $\pi(p)$ the image of 
$p\in\Einhat$ under projection onto $\Eint$.  

The antipodal map being orientation-reversing in the first factor, but
orientation-preserving
in the second, $\Eint$ is non-orientable.  However, it is {\em
time-orientable}, in the sense 
that a future-pointing timelike vector field on $\rtt$ induces one on $\Eint$.

\subsection{Conformally flat Lorentzian structure on $\Eint$}\label{sec:conformal}

Minkowski spacetime may be embedded in an open dense subset of the Einstein Universe.  We shall describe one such embedding in dimension three. 
Let $\rto$ be the 3-dimensional real vector space endowed with a Lorentzian inner product (of signature (2,1)), denoted by $\cdot$.  Set\,:
\begin{align*}
 \iota : \rto & \longrightarrow \Eint \\
\vv & \longmapsto\left( \frac{1-\vv\cdot\vv}{2} :\vv :\frac{1+\vv\cdot\vv}{2}\right) .
\end{align*}
This is a conformal transformation that maps $\rto$ to a
neighborhood of $(1:0:0:0:1)$.  In fact, setting\,:
\begin{equation*}
p_\infty=(-1:0:0:0:1)
\end{equation*}
then \,:
\begin{equation*}
\iota(\rto)=\Eint\setminus\econe{p_\infty}
\end{equation*}
where $\econe{p_\infty}$ is the lightcone at $p$.  (See~\S\ref{sec:photons}.)  Thus $\Eint$ is the {\em conformal compactification of $\rto$}.  

Since $\CONF$ acts transitively on $\Eint$, every point of the Einstein Universe
admits a neighborhood that is conformally equivalent to $\rto$.  In other words,
$\Eint$ is a {\em conformally flat Lorentzian manifold}.

Let $\E$ denote the affine space modeled on $\rto$, that is, the affine space whose underlying vector space of translations 
is endowed with the Lorentzian inner product above.  

By a slight abuse of notation, $\iota$ will also denote the map from $\E$ to $\Eint$ arising from an identification of $\E$ with $\rto$.  In this paper, for $p\in\E$, we will define $\iota(p)$ to be equal to $\iota(\vp)$, where\,: 
\begin{equation}\label{eq:vecp}
\vp=p-(0,0,0).
\end{equation}
Thus $\Eint$ is the conformal compactification of $\E$. 

The embedding $\iota$ induces a monomorphism from the group of Lorentzian similarities into the subgroup of $\SOTT$ which fixes $p_\infty$.  An explicit expression for this map is given in~\cite{BCDGM}.

\subsection{Photons and lightcones}\label{sec:photons}

Let us describe the causal structure of $\Eint$, namely photons and
lightcones.  Conformally
equivalent Lorentzian metrics give rise to the same causal structure (see for instance~\cite{F02}).  As a
matter of fact, the non-parametrized lightlike geodesics are the same.  This
will mean that anything defined in terms of the causal structure of a given
metric will in fact be well defined in the conformal class of that metric.  We will mostly follow here the terminology adopted in~\cite{BCDGM}.

Recall that, given a vector space $V$ endowed with a non-degenerate, symmetric
bilinear form $\ldot{,}$, a subspace of $W\subset V$ is {\em totally isotropic}
if $\ldot{,}$ restricts to an identically zero form on $W$.  In particular, $W\subset\rtt$ is
totally isotropic if and only if $W\setminus\vo\subset\cone$.

\begin{defn}
Let $W\subset\rtt$ be a totally isotropic plane.  Then $\pi(W\setminus\vo)$ is called a
{\em photon}.
\end{defn}
A photon is an unparametrized lightlike geodesic of $\Eint$.  The homotopy 
class of any photon generates the fundamental group of $\Eint$ (with base point chosen on that photon).

\begin{defn}
Two points $p,q\in\Eint$ are said to be {\em incident} if they lie on a common
photon.
\end{defn}

\begin{defn}
Let $p\in\Eint$.  The {\em lightcone} at $p$, denoted $\econe{p}$, is the union
of all photons containing $p$.
\end{defn}
In other words, $\econe{p}$ is the set of all points incident to $p$. 
Also\,:
\begin{equation*}
\econe{p}=\pi\left(\vv^\perp\cap\cone\right)
\end{equation*}
where $\vv\in\cone$ is such that $\pi(\vv)=p$.  

The lightcone $\econe{p}$ is a bouquet of circles (photons) that is pinched at $p$; it is homeomorphic to a pinched torus. Figure~\ref{fig:1cone} shows a lightcone.  
\begin{rem}\label{rem:figures}
For the figures in this paper, we identify $\Eint$ with the quotient of $S^2\times S^1$ by the antipodal map.  We remove a copy of $S^1$, specifically the south pole in each copy of $S^2$, which corresponds to the negative $z$-axis in $\rto$.  Then we cut along a copy of $S^2$ corresponding to the plane $z=0$.  This yields a solid cylinder, where the bottom disk is glued to the top disk in such a way that the former's center gets ``mapped'' to the entire edge of the latter, which is really a point belonging to the removed copy of $S^1$.  See~\S\ref{sec:basic} for the parametrization used.
\end{rem}
\begin{figure}
\centerline{\includegraphics[scale=.5]{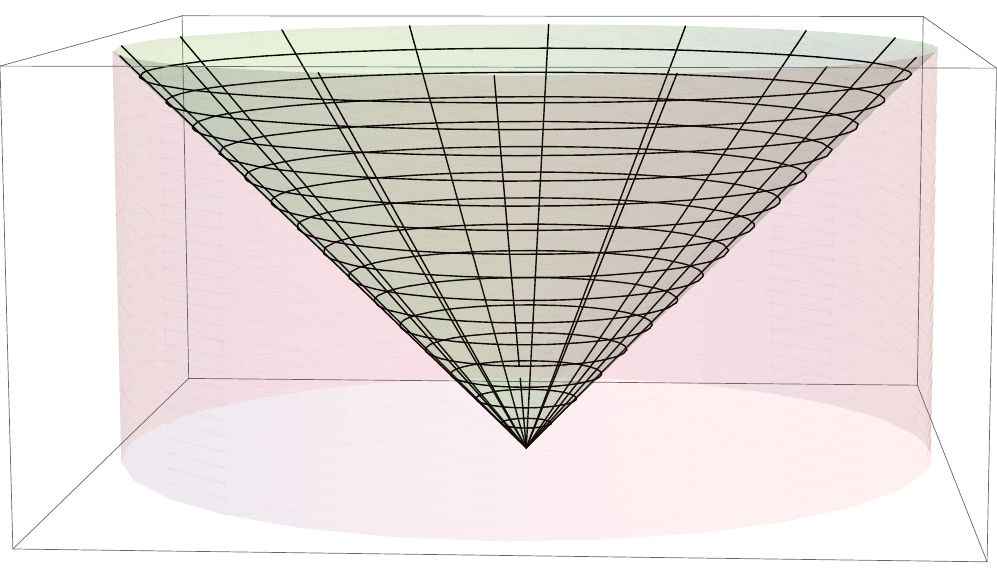}}
\caption{A lightcone in the Einstein Universe.  The point at the bottom is glued to the circle at the top, yielding a pinched torus.  (See Remark~\ref{rem:figures}.)}
\label{fig:1cone}
\end{figure}
Figure~\ref{fig:2cone} shows two lightcones in $\Eint$.  They can be seen to intersect in a simple closed curve.  
\begin{figure}
\centerline{\includegraphics[scale=.5]{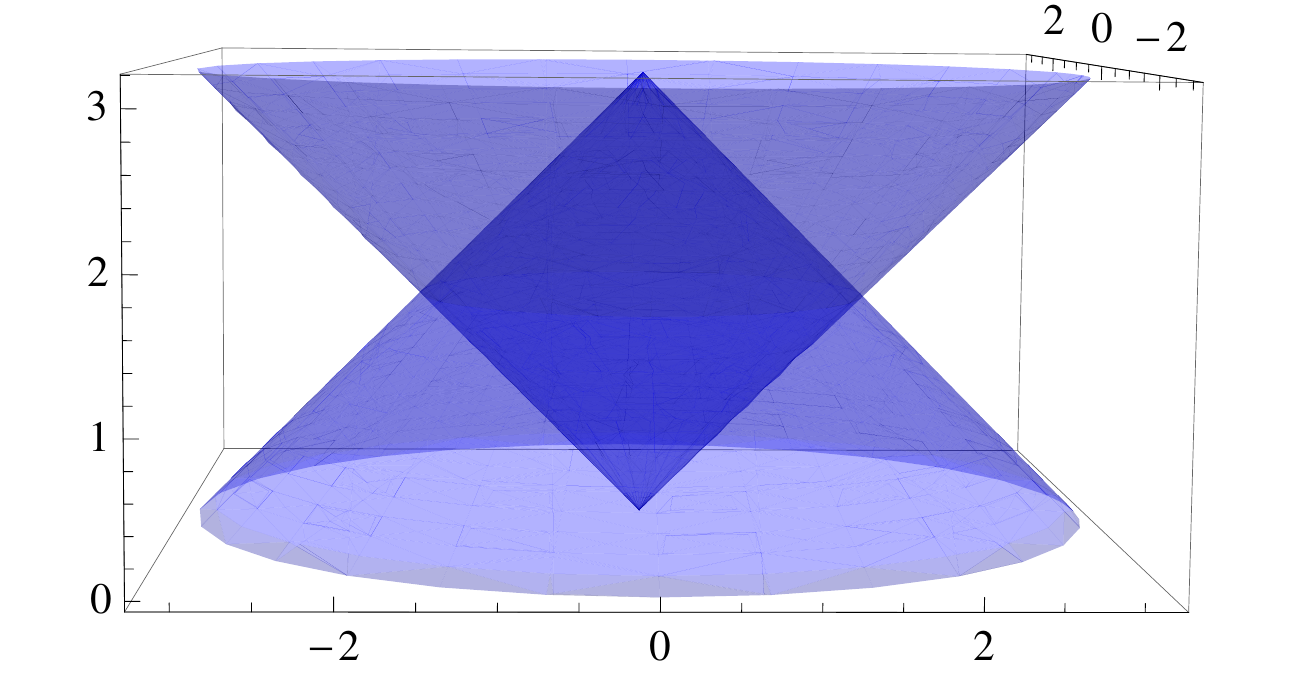}}
\caption{Two lightcones in the Einstein Universe.}
\label{fig:2cone}
\end{figure}
\begin{lemma}\label{lem:twocones}
Suppose $p,q\in\Eint$ are not incident.  Then $\econe{p}\cap\econe{q}$ is a
simple closed curve.  
\end{lemma}

\begin{proof}
Suppose without loss of generality that $p=p_\infty$.  The intersection of $\econe{q}$ with the Minkowski patch $\Eint\setminus\econe{p}$ corresponds to a lightcone in $\rto$.  Applying a translation if necessary, we may suppose that $q=\iota(0,0,0)$.  Then $\econe{p}\cap\econe{q}$ is the so-called circle at infinity:
\begin{equation*}
\econe{p}\cap\econe{q}=\{(0:\cos t:\sin t:1:0)~\mid~t\in\R\}.
\end{equation*}
\end{proof}

\begin{rem}\label{rem:twocones}
The proof of Lemma~\ref{lem:twocones} shows that $\econe{p}\cap\econe{q}$ is spacelike, since the tangent vector at any point of the intersection is spacelike.  The intersection of two such lightcones is called a {\em spacelike circle}.
\end{rem}

\subsection{Einstein torus}\label{sub:torus}

Let $\vv\in\rtt$ be spacelike.  Then the restriction of $\ldot{,}$ to its orthogonal complement $\vv^\perp$ is of signature $(2,2)$.  Its lightcone $\vv^\perp\cap\cone$ projects to a torus in $\Eint$ endowed with a conformal class of metrics of signature $(1,1)$.  It is the conformal compactification of $\R^{1,1}$.

\begin{defn}
An {\em Einstein torus} is the projection $\pi\left(\vv^\perp\cap\cone\right)$, where $\vv\in\rtt$ is spacelike.
\end{defn}
Certain configurations of four points in $\Eint$, called {\em torus data}, induce Einstein tori as follows.  Let\,:
\begin{equation*}
\data=\{p_1,p_2,f_1,f_2\}
\end{equation*}
where\,:
\begin{itemize}
\item $p_1,p_2$ are non-incident;
\item $f_1,f_2\in\econe{p_1}\cap\econe{p_2}$.
\end{itemize}
Let $\vv_1,\vv_2,\vx_1,\vx_2\in\cone$ such
that\,:
\begin{align*}
\vv_i &\in \pi^{-1}(p_i) \\
\vx_i & \in  \pi^{-1}(f_i) .
\end{align*}
The restriction of $\ldot{,}$
endows the subspace of $\rtt$ spanned by the four vectors with a non-degenerate
scalar product of signature $(2,2)$.  Thus it is the orthogonal complement of a spacelike vector in $\rtt$.

%
\section{Crooked surfaces}\label{sec:crookedsurface}

A {\em crooked surface} is an object in the $\SOTT$-orbit of the conformal compactification of  a {\em crooked plane}.  Crooked planes were originally
introduced by Drumm in order to construct fundamental domains for proper affine actions
of free groups  on Minkowski spacetime~\cite{D92}.  They were later generalised
by Frances to study the extension of such actions to $\Eint$~\cite{F03}.

We will start with the definition of a crooked plane, since there is an easily
expressed criterion for disjointness of crooked planes in Minkowski spacetime. 
We will then return to crooked surfaces and give a synthetic description of
these objects in $\Eint$.

\subsection{Crooked planes}
Here is a very brisk introduction to crooked planes.  The reader
interested in more details might consult~\cite{BCDG,DG99}.

Recall that $\E$ denotes the three-dimensional affine space modeled on $\rto$.  To
distinguish between the affine space and the vector space, we will use $o,p$ to
denote points in $\E$ and $\vu,\vx$ to denote vectors in $\rto$.  Given a vector
$\vu\in\rto$, we will denote by $\vu^\perp$ its Lorentz-orthogonal plane\,:
\begin{equation*}
\vu^\perp=\{ \vx\in\rto~\mid~\vu\cdot\vx=0\}.
\end{equation*}

If $\vu\in\rto$ is spacelike, $\vu^\perp$ intersects the lightcone
in two rays.  Let $\xm{\vu}$, $\xp{\vu}$ be a pair of future-pointing null
vectors in $\vu^\perp$ such that $\{ \vu,\xm{\vu},\xp{\vu}\}$ is a positively
oriented basis of $\rto$.

\begin{rem}
There is in fact a ray of possible choices for $\xp{\vu}$, as well as for
$\xm{\vu}$, none of which are more natural than another.  We will typically
write $\vx=\xpm{\vu}$ to mean that they are parallel or, more accurately, that
$\vx$ {\em could be} chosen to be $\xpm{\vu}$.  (But $\vx$ must be future-pointing.)
\end{rem}

If $\vx\in\rto$ is null, then $\vx^\perp$ is a plane, tangent to the
lightcone, and $\vx^\perp\setminus\R\vx$ consists of two connected
components.  In particular, for $\vu$ spacelike, $\vu$ lies in one of these two
components for both $\xm{\vu}$ and $\xp{\vu}$.

\begin{defn}
Let $\vx\in\rto$ be a future-pointing null vector.  Then the closure of the halfplane\,:
\begin{equation*}
\wing{\vx}=\{ \vu\in\vx^\perp~\mid~\vx=\xp{\vu}\}
\end{equation*}
is called a {\em positive linear wing}.

Let $p\in\E$.  A {\em positive wing} is a closed halfplane $p+\wing{\vx}$.
\end{defn}

\begin{rem}
One obtains a {\em negative wing} by choosing the other connected
component of $\vx^\perp\setminus\R\vx$.  In~\cite{BCDG}, positive wings are called ``wings'' and positive linear wings are called ``null halfplanes'', since only positive objects are considered there.
\end{rem}

Observe that if $\vu\in\rto$ is spacelike\,:
\begin{align*}
\vu &\in \wing{\xp{\vu}} \\
-\vu  &\in\wing{\xm{\vu}} \\
\wing{\xp{\vu}}&\cap\wing{\xm{\vu}}  =\{\vo\}.
\end{align*}
The set of positive linear wings is $\SOTO$-invariant.

\begin{defn}
Let $\vu\in\rto$ be spacelike.  Then the following set\,:
\begin{equation*}
\stem{\vu}=\{ \vx\in\vu^\perp~\mid~\vx\cdot\vx\leq 0\}
\end{equation*}
is called a {\em linear stem}.  

Let $p\in\E$.  A {\em
stem} is a set $p+\stem{\vu}$.
\end{defn}
Observe that $\R\xp{\vu}$ and $\R\xm{\vu}$ bound $\stem{\vu}$; 
the closures of each linear wing $\wing{\xp{\vu}}$ and $\wing{\xm{\vu}}$ respectively intersect the linear stem in
these lines.

\begin{defn}
Let $\vu\in\rto$ be spacelike and let $p\in\E$.  The {\em positively extended
crooked plane} with {\em vertex} $p$ and {\em director} $\vu$ is the
union of\,:
\begin{itemize}
\item the stem $p+\stem{\vu}$;
\item the positive wing $p+\wing{\xp{\vu}}$;
\item the positive wing $p+\wing{\xm{\vu}}$.
\end{itemize}
It is denoted $\CP{p,\vu}$.
\end{defn}

Figure~\ref{fig:cp1} shows a crooked plane.
\begin{figure}
\centerline{\includegraphics[scale=.3]{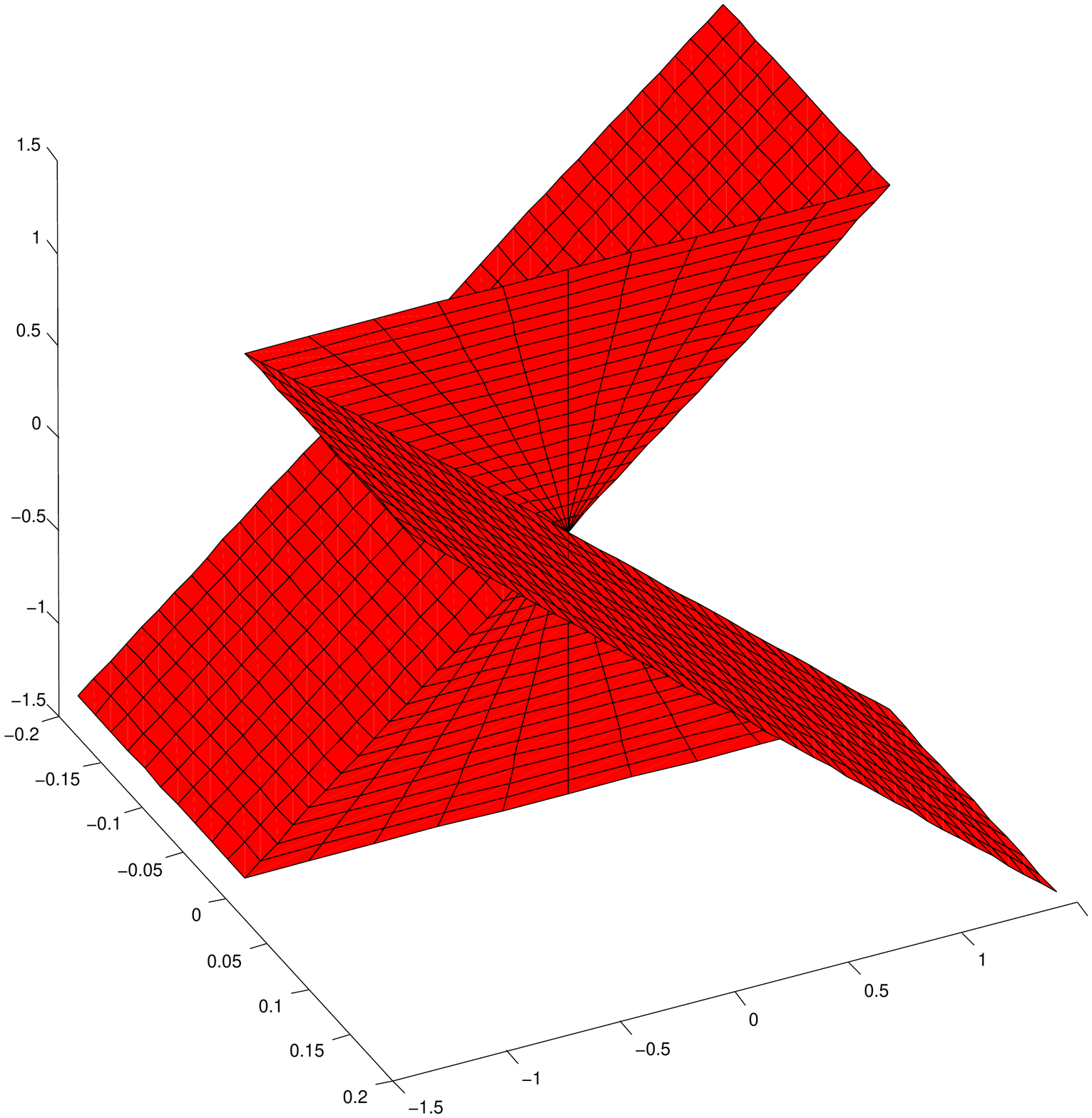}}
\caption{A crooked plane in $\E$.}
\label{fig:cp1}
\end{figure}

\begin{rem}
A negatively extended crooked plane is obtained by replacing positive wings with
negative wings.  We can avoid discussing negatively extended crooked planes,
even though we will make use of their compactifications in~\S\ref{sec:negative}.
 Thus we will simply write ``wing'' to mean positive wing and ``crooked plane"
to mean positively extended crooked plane, until further notice.
\end{rem}

\subsection{Crooked halfspaces and disjointness} 

The complement of a crooked plane $\CP{p,\vu}\in\E$ consists of two crooked halfspaces, respectively corresponding to $\vu$ and $-\vu$.  A crooked halfspace will be determined by the appropriate {\em stem quadrant}, which we introduce next.  In~\cite{BCDG}, the stem quadrant depends on the crooked halfspace, rather than the point-vector pair determining the crooked halfspace as we do here.  But in both cases, a stem quadrant is an affine set determining a unique crooked halfspace.

In what follows, $\int{X}$ denotes the interior of $X$ and $\cl{X}$, the closure of $X$.

\begin{defn}
Let $\vu\in\rto$ be spacelike and let $p\in\E$.  The associated {\em stem quadrant} is\,:
\begin{equation*}
 \Quad{p,\vu}=p+\{ a\xm{\vu}-b\xp{\vu}~\mid~a,b\geq 0\}.
\end{equation*}
 \end{defn}

The stem quadrant $\Quad{p,\vu}$ is bounded by light rays parallel to $\xm{\vu}$ and
$-\xp{\vu}$.

\begin{defn}
Let $\vu\in\rto$ be spacelike and let $p\in\E$.  The {\em crooked halfspace}
$\CHS{p,\vu}$ is the component of the
complement of 
$\CP{p,\vu}$ containing $\int{\Quad{p,\vu}}$.
\end{defn}
By definition, crooked halfspaces are open.  While the crooked planes
$\CP{p,\vu}$, $\CP{p,-\vu}$ are equal, 
the crooked halfspaces $\CHS{p,\vu}$, $\CHS{p,-\vu}$ are disjoint, sharing $\CP{p,\vu}$ as a common boundary.

We will describe disjoint configurations of crooked planes by means of {\em allowable translations} and {\em consistent orientations}.  These are in fact linear notions.  In keeping with the treatment in~\cite{BCDG}, our definitions will be cast in affine terms as well.  We will choose a point $o\in\E$, but the choice of $o$ is completely arbitrary.

\begin{defn}
 Let $\vu_1,\vu_2\in\rto$ be spacelike and choose $o\in\E$.  The vectors $\vu_1,\vu_2$ are said to be
{\em consistently oriented} if $\cl{\CHS{o,\vu_1}}$ and $\cl{\CHS{o,\vu_2}}$ intersect only in $o$.
\end{defn}

The original, equivalent definition, due to Drumm-Goldman~\cite{DG99}, requires among other things that
$\vu_1^\perp\cap\vu_2^\perp$ be spacelike.  Such vectors are called {\em ultraparallel\/}.

Next we describe the displacement vector between vertices of disjoint crooked planes.   This is called an {\em allowable translation} in~\cite{BCDG} and earlier papers.  Here, we really want the decomposition of such an allowable translation as a pair of vectors in the respective stem quadrants and so we introduce a slight modification.
\begin{defn}
Let $\vu_1,\vu_2\in\rto$ be a pair of consistently oriented 
spacelike
vectors.  Choose $o\in\E$ and set\,: 
\begin{equation*}
O=\left(\Quad{o,\vu_1}-o\right)\times\left(\Quad{o,\vu_2}-o\right)\subset\rto\times\rto.
\end{equation*}

The set of {\em allowable pairs} for $\vu_1,\vu_2$ is\,:
\begin{equation*}
 \allow{\vu_1,\vu_2}=\{(\vz_1,\vz_2)\in O~\mid~\vz_1-\vz_2\in \int{\Quad{o,\vu_1}-\Quad{o,\vu_2}}\}.
\end{equation*}
\end{defn}

\begin{rem}
In the case where $\vu_1,\vu_2$ are ultraparallel, $(\vz_1,\vz_2)\in\allow{\vu_1,\vu_2}$ if either $\vz_1$ or $\vz_2$ is parallel to an edge of the relevant stem quadrant, but not both.  Considering $O$ as an orthant in $\R^4$, then $\allow{\vu_1,\vu_2}$ corresponds to the entire polyhedron, minus all but a pair of dimension two faces.  But in the {\em asymptotic} case, where $\vu_1^\perp\cap\vu_2^\perp$ is null, then neither $\vz_1$ nor $\vz_2$ can be parallel to the common edge of the stem quadrants.  
\end{rem}

Drumm and Goldman~\cite{DG99} proved a criterion for disjointness of crooked planes in terms of allowable translations.  This criterion is strengthened in~\cite{BCDG} and we adapt the latter formulation to the language of allowable pairs.

\begin{thm}
\cite{BCDG}
\label{thm:strong}
 Let $\vu_1,\vu_2\in\rto$ be a pair of consistently oriented 
spacelike
vectors.  Let $p\in\E$ and let $\vz_1,\vz_2\in\rto$ such that $p+\vz_i\in\Quad{p,\vu_i}$, for $i=1,2$.   Then\,:
\begin{equation*}
 \CP{p+\vz_i,\vu_i}\subset \cl{\CHS{p,\vu_i}}.
\end{equation*}
Moreover, if 
$(\vz_1,\vz_2)\in\allow{\vu_1,\vu_2}$ then\,:
\begin{equation*}
 \CP{p+\vz_1,\vu_1}\cap \CP{p+\vz_2,\vu_2}=\emptyset.
\end{equation*}
\end{thm}

\subsection{Crooked surfaces as conformal compactifications of crooked planes}

Recall from~\S\ref{sec:conformal} that $\iota(\E)$ consists of the
complement of $\econe{p_\infty}$, where $p_\infty=\left(-1:0:0:0:1\right)$.

Let $\vu\in\rto$ be spacelike and let $p\in\E$.  The crooked plane $\CP{p,\vu}$
admits a 
conformal compactification, which we denote by $\conf{\CP{p,\vu}}$. 
Explicitly\,: 
\begin{equation*}
 \conf{\CP{p,\vu}}=\iota(\CP{p,\vu})\cup\phi_\infty\cup\psi_\infty
\end{equation*}
where\,:
\begin{itemize}
 \item $\phi_\infty\subset\econe{p_\infty}$ is the photon containing
$(0:\xp{\vu}:0)$;
  \item $\psi_\infty\subset\econe{p_\infty}$ is the photon containing
$(0:\xm{\vu}:0)$.
\end{itemize}
The pair of photons $\{\phi_\infty,\psi_\infty\}$ forms the ``scaffolding'' of the crooked surface with another pair of photons $\{\phi_p,\psi_p\}$, defined as follows\,: 
\begin{align*}
\phi_p & = \cl{\iota\left(p+\R\xp{\vu}\right)}= \iota\left(p+\R\xp{\vu}\right)\cup\left\{\left(-\vp\cdot\xp{\vu}:\xp{\vu}:\vp\cdot\xp{\vu}\right)\right\}\\
\psi_p & =\cl{\iota\left(p+\R\xm{\vu}\right)}= \iota\left(p+\R\xm{\vu}\right)\cup\left\{\left(-\vp\cdot\xm{\vu}:\xm{\vu}:\vp\cdot\xm{\vu}\right)\right\}
\end{align*}
where $\vp=p-(0,0,0)$ as in Equation~\eqref{eq:vecp}.  We observe the following intersections\,:
\begin{align*}
\phi_\infty\cap\phi_p &= \left\{\left(-\vp\cdot\xp{\vu}:\xp{\vu}:\vp\cdot\xp{\vu}\right)\right\} \\
\psi_\infty\cap\psi_p & = \left\{\left(-\vp\cdot\xm{\vu}:\xm{\vu}:\vp\cdot\xm{\vu}\right)\right\}.
\end{align*}
See Figure~\ref{fig:cp}.
\begin{figure}
\centerline{\includegraphics[scale=.5]{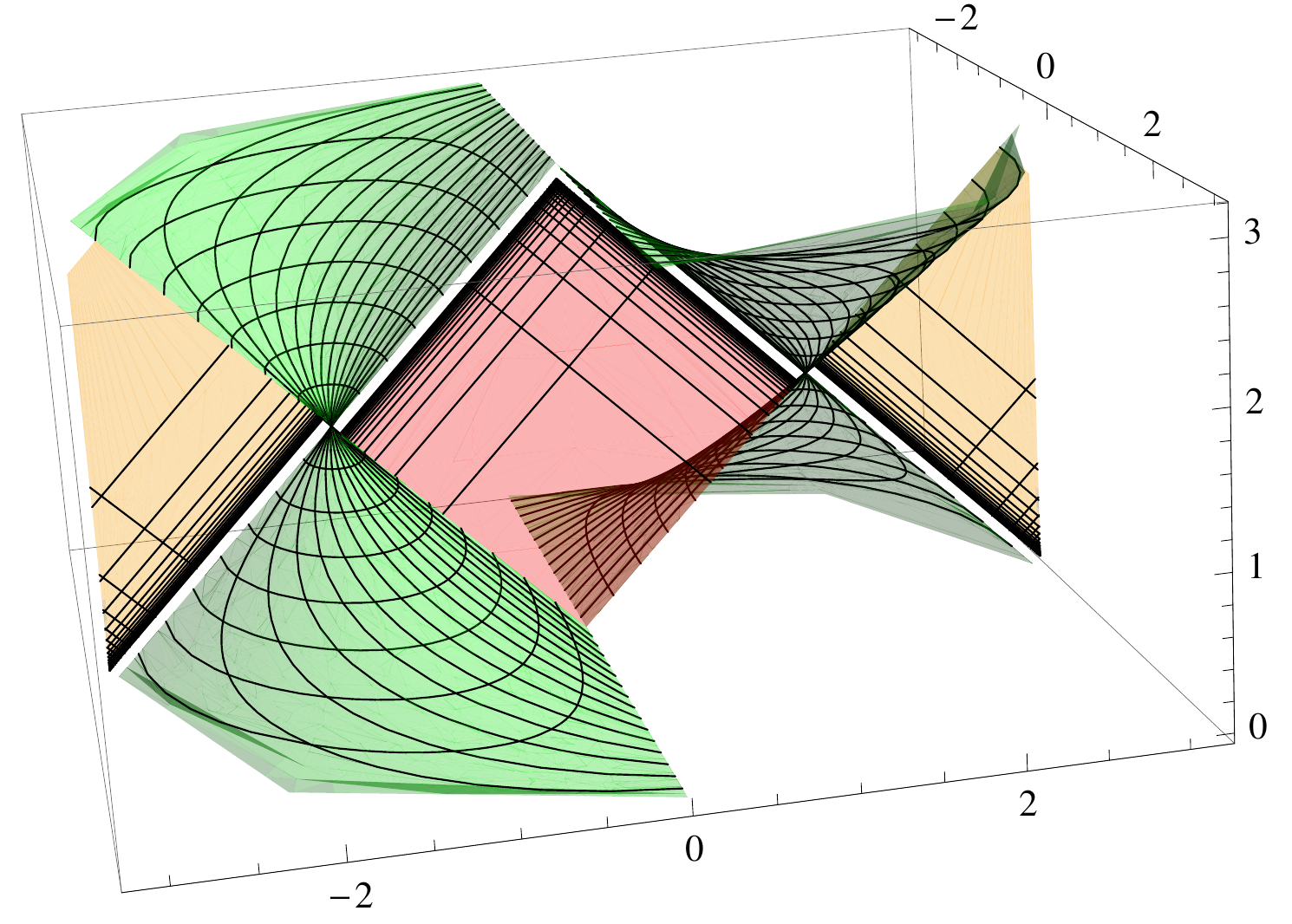}}
\caption{A crooked surface.  The `` flat'' pieces form the stem. The stem appears to be cut in half, because of the removal of a circle (see Remark~\ref{rem:figures}).  The remaining two pieces are the wings.}
\label{fig:cp}
\end{figure}

The compactification of the wing $p+\wing{\xp{\vu}}$ is in fact a ``half lightcone''.  For instance, if $p=(0,0,0)$, then\,:
\begin{equation*}
\phi_p\subset\econe{0:\xp{\vu}:0}.
\end{equation*}
There are two connected components in $\econe{0:\xp{\vu}:0}\setminus\left(\phi_\infty\cup\phi_p\right)$, one of which contains the interior of $\iota\left(p+\wing{\xp{\vu}}\right)$.  The analogous statement holds for  $p+\wing{\xm{\vu}}$ as well.

\begin{defn}
Let $o\in\E$ and $\vu\in\rto$ be spacelike.  A {\em crooked surface} is any
element in the $\SOTT$-orbit of $\conf{\CP{o,\vu}}$.
\end{defn}

\begin{rem}
The conformal compactifications of negatively extended
crooked planes also lie in the $\SOTT$-orbit of $\conf{\CP{o,\vu}}$.  Restricting to the 
connected component of the identity of $\SOTT$, however, yields only what
one could 
call a ``positively extended crooked surface''. 
\end{rem}

\subsection{A basic example}\label{sec:basic}

We will describe  $\CS=\conf{\CP{o,\vu}}$, where $o=(0,0,0)$ and $\vu=(1,0,0)$.  For the remainder of the paper, set\,:
\begin{equation*}
p_0=\iota(o)=(1:0:0:0:1).
\end{equation*}
Note also that\,: 
\begin{equation*}
\iota(\vu)=(0:1:0:0:1).
\end{equation*}

Identifying $\Eint$ with a quotient of $S^2\times S^1$, consider $\Eint$ with the following (non-injective) parametrization\,:
\begin{align}\label{param}
( \cos \phi :\sin\phi\cos\theta&:\sin\phi\sin\theta:\sin t:\cos t) \\
\notag 0&\leq \phi\leq\pi \\
\notag 0&\leq \theta \leq 2\pi \\
\notag 0&\leq t \leq\pi
\end{align}
where $t=0$ is glued to $t=\pi$ by the antipodal map.  In our permuted version of the usual parametrizations of $S^2$ and $S^1$, the point $p_0$ corresponds to $t=\phi=0$ (with $\theta$ arbitrary).  

Since $\xpm{\vu}=(0,\mp 1,1)$, the compactification of $\iota\left(\vu^\perp\right)$ is the Einstein torus with torus data $\{ p_0,p_\infty,f_1,f_2\}$, where \,:
\begin{align*}
f_1 &=(0:0:1:1:0)\\
f_2 &=(0:0:-1:1:0).
\end{align*}
This is the torus\,:
\begin{equation*}
 \pi\left((0,1,0,0,0)^\perp\cap\cone\right)
\end{equation*}
which corresponds to $\theta=\pm\pi/2$ in the parametrization~\eqref{param}.  We obtain the stem by restricting $\phi$ to lie outside of the open interval bounded by $t$ and $\pi/2-t$.

%
%
The wing $o+\wing{\xm{\vu}}$ projects into $\econe{0:0:1:1:0}$.  Compactifying yields the half lightcone bounded by the pair of photons $\phi=t$ and $\phi=\pi/2-t$, containing the point $(0:-1:0:0:1)$. Its intersection with the sphere $t=t_0$ is a circle, with center $(0:0:1:1:0)$ and (spherical) radius $\pi/2-t_0$.

In the same manner, the wing $o+\wing{\xp{\vu}}$ corresponds to the half lightcone in $\econe{0:0:-1:1:0}$, bounded by the pair of photons $\phi=t$ and $\phi=\pi/2-t$ and containing the point $(0:1:0:0:1)$. 

Alternatively, parametrize the wing $o+\wing{\xm{\vu}}$ as follows\,:
\begin{align*}
(\sin s\cos t  :\sin s\sin t&:\cos s:\cos s:-\sin s) \\
 - \frac{\pi}{2} &\leq s\leq \frac{\pi}{2}\\
 0 &\leq t\leq 2\pi.
\end{align*}
Now photons correspond to $t=\mbox{constant}$.  The wing is the half lightcone $0\leq t\leq\pi$.

Given that all crooked surfaces are conformally equivalent to the basic example, the following two theorems are easily verified.  (See also Figure~\ref{fig:cp}.)

\begin{thm}
 A crooked surface is a surface homeomorphic to a Klein bottle.
\end{thm}

\begin{proof}
This is a simple cut-and-paste argument, as can be found in~\cite{BCDGM}.
\end{proof}

\begin{thm}
 A crooked surface separates $\Eint$.
\end{thm}

\begin{proof}
It suffices to show this for $\CS$ in the basic example.

Let $\hat{\CS}$ be the inverse image of $\CS$ in $S^2\times [0,\pi]\subset\Einhat$.  Let us first examine the intersection of $\hat{\CS}$ with $S^2\times\{t\}$, for $t\in[0,\pi]$.   
\begin{itemize}
\item When $t=0$, this is the union of the two meridians $\theta=0$ and $\theta=\pi$, each meridian belonging to a wing.  
\item When $t=\pi$, we obtain the same two meridians, but they now belong to the opposite wing.  
\item When $t=\pi/2$, we obtain the union of the two meridians $\theta=\pm\pi/2$.
\item For all other values of $t$, the intersection consists of two arc segments on $\theta=\pm\pi/2$, joined by two half circles (one for each wing).
\end{itemize}
Thus for every $t\in[0,\pi]$, $\hat{\CS}\cap (S^2\times\{t\})$ separates $S^2\times\{t\}$.  Consequently, since $\hat{\CS}$ is a connected surface (with boundary), it separates $S^2\times [0,\pi]$.  Furthermore, the following two subsets of $S^2\times [0,\pi]\setminus\hat{\CS}$ belong to the same connected component\,:
 \begin{align*}
 \{0<\theta<\pi\} & \times\{t=0\} \\
  \{\pi<\theta<2\pi\} & \times\{t=\pi\}.
  \end{align*}
  These two sets are mapped to each other by the antipodal map.   Therefore, $\Eint\setminus\CS$ has two connected components as well.
%
%
\end{proof}

\subsection{Torus data spanning a crooked surface}

A crooked surface in $\Eint$ is determined by torus data, as we will now see.  Set\,: 
\begin{equation*}
 \data=\{p_1,p_2,f_1,f_2\}
\end{equation*}
where, as in~\S\ref{sub:torus}\,: 
\begin{itemize}
\item $p_1,p_2\in\Eint$ are a pair of non-incident points ;
\item $f_1,f_2\in\econe{p_1}\cap\econe{p_2}$.
\end{itemize}

The torus data $\data$ determines two pairs of photons\,:
\begin{itemize}
\item one pair of photons $\phi_1,\phi_2$ belonging to $\econe{p_1}$;
\item one pair of photons $\psi_1,\psi_2$ belonging to $\econe{p_2}$;
\item each pair of photons $\phi_i,\psi_i$ intersects in $f_i$.  
\end{itemize}
Alternatively, $\phi_i$ and $\psi_i$ belong to the lightcone of $f_i$.

Here is how we may define the stem of the crooked surface.  Choose any conformal map\,: 
\begin{equation*}
\kappa:\Eint\setminus\econe{p_2}\longrightarrow\rto
\end{equation*}
and set\,: 
\begin{equation*}
 \NF{p_1}=\{ q\in\Eint\setminus\econe{p_2}~\mid~\kappa(q)-\kappa(p_1)\mbox{ is a timelike vector}\}.
\end{equation*}
(Alternatively, $\NF{p_1}$ is the set of endpoints of 
timelike geodesic paths starting from $p_1$, given an arbitrary choice of Lorentz metric in the conformal class.)  Observe that the definition is symmetric in the pair $\{p_1,p_2\}$, so that we can write unambiguously\,:
\begin{equation*}
\NF{\data}=\NF{p_1}=\NF{p_2}.
\end{equation*}
Let $\einp{\data}$ be the Einstein torus determined by the torus data $\data$.  The
stem, denoted $\stem{\data}$, consists of the intersection\,:
\begin{equation*}
\stem{\data}=\einp{\data}\cap\NF{\data}.
\end{equation*}
%



As for the wings, consider the following.  The pair of photons $\phi_1$ and
$\psi_1$ bounds two half lightcones in $\econe{f_1}$, as does the pair
$\phi_2$ and $\psi_2$ in $\econe{f_2}$. Each half lightcone intersects the
common spacelike circle in a half circle, bounded by $p_1$ and $p_2$. We want
to choose one half lightcone from each pair, in such a way that they intersect
only in $p_1$ and $p_2$.  
In fact, there are two ways to do this.  Conformally identifying
$\Eint\setminus\econe{p_2}$ with $\iota(\rto)$ via an element of $\sott$\,:  
\begin{itemize}
 \item one way yields a pair of positive wings, $\wing{\data,f_1}$ and
$\wing{\data,f_2}$;
  \item the other way yields a pair of negative wings.
\end{itemize}

Thus the crooked surface {\em spanned by the torus data} $\data$ is\,:
\begin{equation*}
\CS\left(\data\right)= \stem{\data}\cup\wing{\data,f_1}\cup\wing{\data,f_2}.
\end{equation*}

\begin{lemma}\label{lem:nocontain}
Any lightcone in $\Eint$ intersects both components of the complement of a
crooked surface.
\end{lemma}

\begin{proof}
Let $q\in\Eint$ and consider the crooked surface spanned by the torus data $\data=\{p_1,p_2,f_1,f_2\}$.  If $q,f_1$ are not incident, then, $\econe{q}$ and $\econe{f_1}$ intersect transversely, by Lemma~\ref{lem:twocones}
.  The analogous statement holds if $q,f_2$ are not incident.  

Otherwise, $q$ lies on $\econe{f_1}\cap\econe{f_2}$, so that its lightcone shares a photon $\phi$ with either $\wing{\data,f_1}$ or $\wing{\data,f_2}$; say $\phi=\econe{q}\cap\wing{\data,f_1}$.  In that case, any photon in $\econe{q}$ that is different from $\phi$ intersects $\wing{\data,f_1}$ transversely. 
\end{proof}

\section{How to get disjoint crooked surfaces}\label{sec:disjoint}

We have seen how to obtain disjoint crooked planes, starting with a pair of crooked planes sharing a vertex and ``pulling them apart'' using allowable translations, or 
translations in their respective stem quadrants.  We will promote this idea to crooked surfaces.  We will apply the observation that the torus data spanning a crooked surface is symmetric in the first pair of points (or even the second pair of points).  We can pull a pair of crooked surfaces apart at both points, producing a pair of disjoint crooked surfaces.

\subsection{Step one\,: start with an appropriate pair of crooked planes and
compactify}

Let $\vu_1,\vu_2\in\rto$ be a pair of consistently oriented spacelike
vectors.  To fix ideas, set $o=(0,0,0)\in\E$.  The crooked surfaces $\conf{\CP{o,\vu_i}}$ 
intersect in exactly two points, namely, $p_0$ and $p_\infty$.  Express the pair of crooked surfaces as $\CS(\data_i)$, for $i=1,2$, where\,: 
\begin{align*}
\data_1 &=\{p_0,p_\infty,f_1,f_2\} \\
\data_2 &=\{p_0,p_\infty,f'_1,f'_2\}
\end{align*}
and $f_1,f_2,f'_1,f'_2$ are distinct points in $\econe{p_0}\cap\econe{p_\infty}$.

Set\,: 
\begin{equation*}
H_i=\conf{\CHS{o,\vu_i}},~i=1,2.
\end{equation*}
These are the components of the complement of each
$\CS(\data_i)$ that are disjoint.  Their closures intersect in exactly $p_0$ and $p_\infty$.

\begin{lemma}\label{lem:stepone}
 Let $(\vz_1,\vz_2)\in\allow{\vu_1,\vu_2}$.  Then\,: 
\begin{equation*}
 \conf{\CP{o+\vz_i,\vu_i}}\subset\cl{H_i}
\end{equation*}
Moreover, these crooked surfaces intersect only in $p_\infty$.
\end{lemma}
\begin{proof}
 By Theorem~\ref{thm:strong}, $\CP{o+\vz_i,\vu_i}\subset\CHS{o,\vu_i}$ and thus
its conformal 
compactification remains in the closure of $H_i$. However, the pair of crooked
surfaces will no 
longer intersect in $p_0$.
\end{proof}

\subsection{Step two\,: do the same thing at infinity}

Given a crooked surface spanned by torus data $\{p_1,p_2,f_1,f_2\}$, there exists a conformal involution of $\Eint$ that permutes $p_1$ and $p_2$, fixes $f_1$ and $f_2$ and leaves the crooked surface invariant.  In the case at hand, set\,: 
\begin{equation*}
\rr : (v_1:v_2:v_3:v_4:v_5)\longmapsto (-v_1:v_2:v_3:v_4:v_5).
\end{equation*}
Then $\rr(p_0)=p_\infty$; moreover, $\rr$ pointwise 
fixes $\econe{p_0}\cap\econe{p_\infty}$.  In fact, its fixed point set is the Einstein torus\,: 
\begin{align*}
(0 & : \sin t:\cos t : \sin s: \cos s) \\
0 & \leq t\leq 2\pi \\
0 & \leq s\leq \pi.
\end{align*}
Since it is a conformal map, $\rr$ maps stems and wings, respectively, to stems and wings.
\begin{lemma}\label{lem:nuinvariant}
 Every crooked surface spanned by torus data $\{p_0,p_\infty,*,*\}$ is
$\rr$-invariant.  Each of its complementary components is invariant as well.
\end{lemma}
\begin{proof}
Conjugating by an automorphism of $\Eint$, it suffices to check this for $\data$ as in~\S\ref{sec:basic}.  The stem is invariant, since the underlying Einstein torus is.  As for the wings, they must be invariant as well since they intersect the fixed point set, as does each component of the complement of the crooked surface.
\end{proof}

For any $\vv\in\rto$, let $\tt_\vv$ be the linear automorphism of $\rtt$
induced by 
translation in $\E$ by $\vv$.  This induces an automorphism of $\Eint$ which we
abusively
 also denote by $\tt_\vv$.  In particular, $\rr\tt_\vv\rr$ is an
automorphism of
$\Eint$ which fixes $p_0$.
The following is an immediate consequence of Lemmas~\ref{lem:stepone} and~\ref{lem:nuinvariant}.
\begin{cor}
 Let $(\vz_1,\vz_2)\in\allow{\vu_1,\vu_2}$.  Then $\rr\tt_{\vz_i}\rr$
maps 
$\conf{\CP{o,\vu_i}}$ into $\cl{H_i}$.  In particular, the resulting pair
of 
crooked surfaces intersect in $p_0$ only.
\end{cor}

\begin{thm}\label{thm:disjointcs}
 Let $(\vz_1,\vz_2),(\vz_1',\vz_2')\in\allow{\vu_1,\vu_2}$.  Then the crooked
surfaces 
$\rr\tt_{\vz'_i}\rr\,\conf{\CP{o+\vz_i,\vu_i}}$ are disjoint.
\end{thm}

\begin{proof}
For $i=1,2$, both $\tt_{\vz_i}$ and $\rr\tt_{\vz'_i}\rr$ leave $H_i$ invariant and each removes a point of intersection between the crooked surfaces.
\end{proof}

Figure~\ref{fig:twocp} shows a pair of disjoint crooked surfaces.
\begin{figure}
\centerline{\includegraphics[scale=.5]{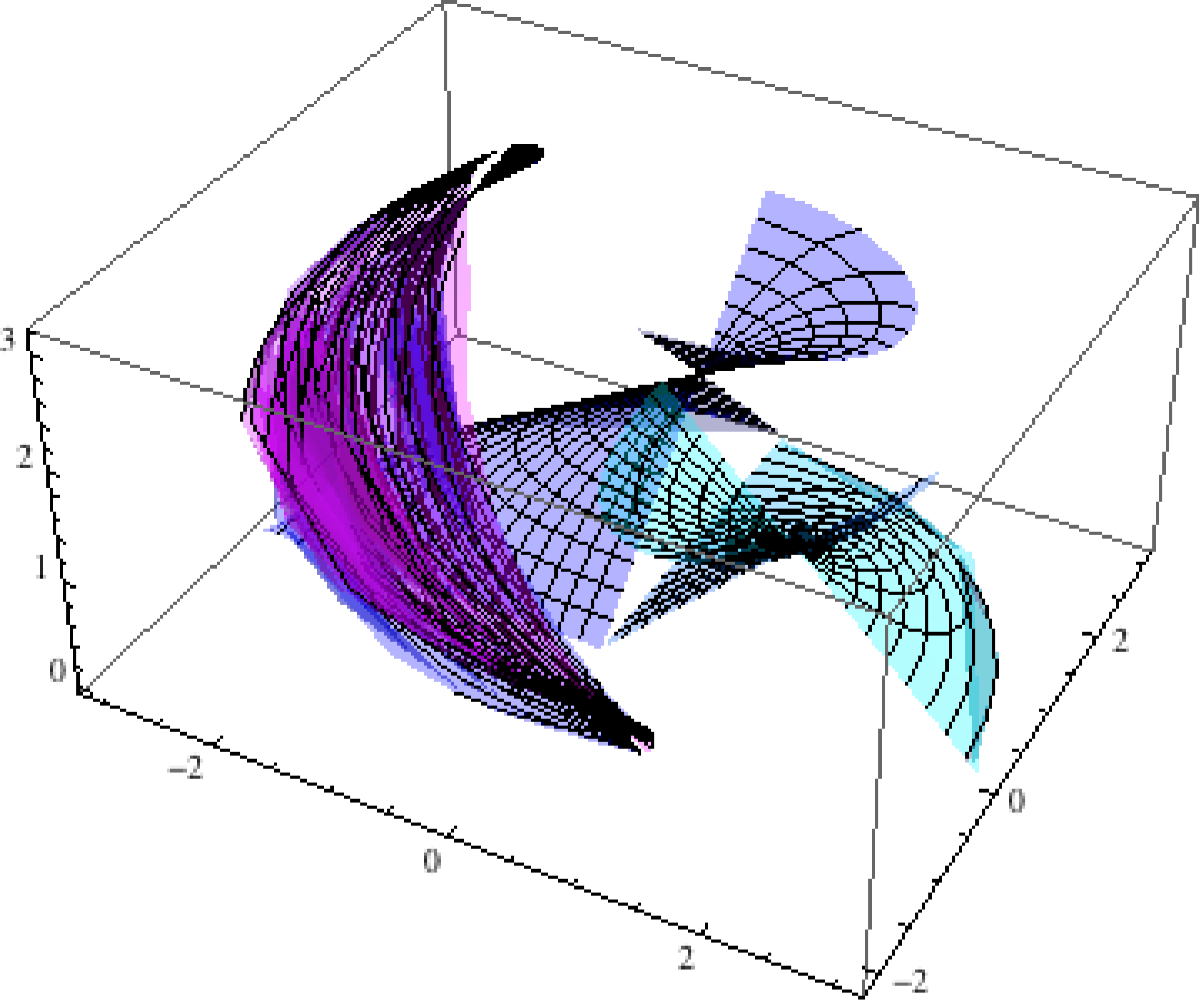}}
\caption{A pair of disjoint crooked surfaces.  Only the wings are shown (in order to keep the picture file a decent size).  The rightmost and center objects are the wings of one crooked surface; the left-most image is in fact a pair of wings, for the second crooked plane.}
\label{fig:twocp}
\end{figure}
%
%

\section{Schottky groups}\label{sec:schottky}
The procedure described in \S\ref{sec:disjoint} can be applied to obtain any
number of 
pairwise disjoint crooked surfaces and more precisely, pairwise disjoint regions
bounded by crooked 
surfaces.   In this section, we will show that our construction yields examples closely related to the {\em Lorentzian Schottky groups} introduced
and studied by Frances in~\cite{F02,F05}.  

\begin{defn}
We say that $\Gamma=\langle\hh_1,\ldots,\hh_n\rangle<\SOTT$ admits a {\em
crooked Schottky domain} if there exist $2n$ 
pairwise disjoint crooked surfaces $\CS_1^\pm,\ldots,\CS_n^\pm$, bounding $2n$
pairwise disjoint (open) regions $U_1^\pm,\ldots,U_n^\pm$, 
such that for each $i=1,\ldots,n$\,:
\begin{equation}\label{eq:pingpong}
 \hh_i(U_i^-)=\Eint\setminus\cl{U_i^+} .
\end{equation}
\end{defn}

The previous section suggests a way of constructing such groups.  To keep the discussion as simple as possible, we will describe a cyclic example.  

Let $\gg\in\SOoTO$ be a hyperbolic element (that is, $\gg$ has three distinct real eigenvalues).  Acting by isometries on the hyperbolic plane $\HP$, the cyclic free group $\langle\gg\rangle$ admits a fundamental domain bounded by two disjoint lines, corresponding to a pair of crooked planes $\CP{o,\vu_i}$, where $\vu_1,~\vu_2$ are consistently oriented.   Slightly abusing notation, denote also by $\gg$ the element in $\SOTT$ that fixes $p_\infty$ and whose action on $\iota(\rto)$ is the same as the action of $\gg$ on $\rto$.  Now choose $(\vz_1,\vz_2),~(\vz_1',\vz_2')\in\allow{\vu_1,\vu_2}$.  Set\,: 
\begin{align*}
\tt_1 & =\rr\tt_{\vz'_1}\rr\tt_{\vz_1} \\
\tt_2 & = \rr\tt_{\vz'_2}\rr\tt_{\vz_2}.
\end{align*}
Then\,: 
\begin{equation*}
 \tt_2\gg\tt_1^{-1}(U^-)=\Eint\setminus\cl{U^+} .
\end{equation*}
where\,:
\begin{align*}
U^- & = \rr\tt_{\vz'_1}\rr\conf{\CHS{o+\vz_1,\vu_1}} \\
U^+ & = \rr\tt_{\vz'_2}\rr\conf{\CHS{o+\vz_2,\vu_2}}
\end{align*}
which are disjoint, by Theorem~\ref{thm:disjointcs}.  (Observe that for $\vv\in\rto$, we have $\tt_\vv^{-1}=\tt_{-\vv}$ as expected.)

Let $\Gamma=\langle\hh_1,\ldots,\hh_n\rangle$ be a subgroup of $\SOTT$ admitting a crooked Schottky domain, bounded by $U_1^\pm,\ldots,U_n^\pm$.  We will see that the action of $\Gamma$ is very much like the usual action of a Schottky group acting on a sphere (or the hyperbolic plane).

First, consider each cyclic subgroup $\langle\hh_i\rangle$.  Set\,: 
\begin{equation*}
F_i=\E\setminus\left(U_i^-\cup U_i^+\right).
\end{equation*}
Clearly, $\langle\hh_i\rangle$ acts properly on the open set of all $\Gamma$-translates of $F_i$.  Since the $U_i^\pm$ are pairwise disjoint, we may apply the Klein-Maskit Combination Theorem to conclude that the free group $\Gamma$ acts properly on the following open set\,:
\begin{equation*}
\Omega=\bigcup_{\gg\in\Gamma}\gg(F)
\end{equation*}
where\,: 
\begin{equation*}
F=\bigcap_{i=1}^nF_i.
\end{equation*}

Moreover, $\Omega/\Gamma$ is compact.  Indeed, it is obtained by removing the open sets $U_1^\pm,\ldots,U_n^\pm$ from $\Eint$, which is compact, and then gluing together the boundaries of each pair $U_i^\pm$.  Therefore, $\Gamma$ is a {\em Lorentzian Kleinian group}~\cite{F02,F05}.

The complement of $\Omega$ also resembles the limit set of an ordinary Schottky group.  There is a one-to-one correspondence (in fact, a homeomorphism) between the set of connected components of $\Eint\setminus\Omega$ and the limit set of $\Gamma$.  An infinite word $\hh_{i_1}^{s_1}\hh_{i_2}^{s_2}\ldots$ corresponds to the Hausdorff limit of the following sequence of compact sets\,:
\begin{equation}\label{eq:nested}
\hh_{i_1}^{s_1}\ldots\hh_{i_{n-1}}^{s_{n-1}}\cl{U_{i_n}^{s_n}}
\end{equation}
where $s_i=\pm 1$ or $\pm$.

\begin{rem}
We make two remarks on notation which will be tacitly assumed from now on.  The first is that when we write a word in terms of the $\hh_j$'s, the superscripts $s_j$ will always be either 1 or -1; but   we will always assume the words to be {\em reduced}, meaning that no $\hh_j^{s_j}$ is followed by its inverse.  The second is that for sets appearing in a Schottky configuration, the superscripts will always be either $+$ or $-$; for example, if $s=-$, then $-s=+$ etc.
\end{rem}

\subsection{Conformal dynamics}

The group $\Gamma$ is {\em almost} what Frances calls a Lorentzian Schottky group.  The only missing ingredient is a condition on the dynamics of the generators $\hh_i$, which we now discuss.  Our presentation is admittedly terse and we refer the reader to~\cite{F05} and~\cite{BCDGM} for more details. 

The semisimple Lie group $\SOTT$ admits a Cartan decomposition $\SOTT=KAK$, where $K$ is a maximal compact subgroup and $A$ is a Cartan subgroup, which we may take as consisting of diagonal matrices of the following form\,:

\begin{equation}\label{eq:dynamics}
\begin{bmatrix}
         e^{\lambda} & 0 & 0 & 0 & 0 \\
	 0 & e^{\mu} & 0 & 0 & 0 \\
	 0 & 0 & 1 & 0 & 0 \\
	 0 & 0 & 0 & e^{-\mu} & 0 \\
	 0 & 0 & 0 & 0 & e^{-\lambda}
        \end{bmatrix}
\end{equation}
where $\lambda>\mu>0$.

Let $\{\gg_n\}$ be a sequence in $\SOTT$; then we may write $\gg_n=\xi_n\alpha_n\xi'_n$, where $\xi_n,\xi'_n\in K$ and $\alpha_n\in A$.  Denote by $\lambda_n,\mu_n,$ the exponents appearing in the first two diagonal terms of $\alpha_n$, as in~\eqref{eq:dynamics}.  Set $\delta_n=\lambda_n-\mu_n$.

Suppose now that $\{\gg_n\}$ is a divergent sequence that {\em tends simply to infinity}, that is, the sequences $\{\xi_n\}$, $\{\xi'_n\}$ converge in $K$ and\,: 
\begin{align*}
\lambda_n & \longrightarrow \lambda_\infty \\
\mu_n & \longrightarrow \mu_\infty \\
\delta_n & \longrightarrow \delta_\infty
\end{align*}
where $\lambda_\infty,\mu_\infty,\delta_\infty\in\R\cup\{\infty\}$.  Then $\{\gg_n\}$ is said to have\,:
\begin{itemize}
\item {\em bounded distortion} if $\mu_\infty\in\R$;
\item {\em balanced distortions} if $\lambda_\infty=\mu_\infty=\infty$ and $\delta_\infty\in\R$;
\item {\em mixed distortions} if $\lambda_\infty=\mu_\infty=\delta_\infty=\infty$ .
\end{itemize}
Of particular interest are groups which contain no sequences of bounded distortion, for they act properly on 4-dimensional anti-de Sitter space, of which $\Eint$ is the boundary~\cite{F05}.

For $z\in\Eint$, we define its {\em dynamical set} (relative to the sequence $\{\gg_n\}$) as follows\,:
\begin{equation*}
D(z)=\bigcup_{z_n\rightarrow z}\{\mbox{ accumulation points of } \{\gg_n(z_n)\}\}.
\end{equation*}
For a set $X\subset\Eint$, we define $D(X)$ to be the union of the dynamical sets of its elements.  

A group acts properly on a set $\Omega$ if and only if no two points in $\Omega$ are in each other's dynamical sets.

Sequences of bounded distortion are characterized by the existence of two points $p^+,~p^-\in\Eint$, respectively called the {\em attractor} and the {\em repellor}, whose associated lightcones, $\econe{p^\pm}$ have the following dynamical property.  There exists a diffeomorphism $\mm$ between the set of photons in $\econe{p^-}$ and $\econe{p^+}$, such that\,: 
\begin{itemize}
\item for every $z\in\Eint\setminus\econe{p^-}$, $D(z)=\{p^+\}$;
\item for every photon $\phi\subset\econe{p^-}$ and every $z\neq p^-\in\phi$, $D(z)=\mm(\phi)$;
\item $D(p^-)=\Eint$.
\end{itemize}
Roughly speaking, to obtain a set on which a sequence of bounded distortion acts properly, one must remove at least one of the two photons $\phi,\mm(\phi)$, for each $\phi\subset\econe{p^-}$.  In contrast, sequences with balanced or mixed distortions require the removal of only a pair of photons.  In particular, Lorentzian Schottky groups, which are generated by elements with mixed distortions, act properly on the complement of a Cantor set of photons.

While we cannot say whether our examples must have mixed distortions, or whether they might admit balanced distortions, we can nevertheless prove the following.

\begin{thm}
Let $\Gamma<\SOTT$ admit a crooked Schottky domain.  Then $\Gamma$ does not admit any sequences with bounded distortion.
\end{thm}

Our argument is modeled on the proof of Lemma 8 in~\cite{F05}, which shows that every connected component of the limit set of a Lorentzian Schottky group consists of a single photon.  (The argument does not require that the generators of $\Gamma$ have mixed dynamics.)  We refer the reader to that paper for more details.

\begin{proof}
Let $\hh_1,\ldots,\hh_n$ be a set of free generators of $\Gamma$ and let $U^\pm_1,\ldots,U^\pm_n$ be the open sets in the complement of the crooked fundamental domain.  Suppose that $\{\gg_j\}\subset\Gamma$ is a sequence with bounded distortion.  We will denote its attractor and repellor by $p^+,p^-$, respectively and the diffeomorphism matching photons in $\econe{p^\pm}$ by $\mm$.

Write each $\gg_j$ as a reduced word in the free generators.  Taking a subsequence if necessary, we may assume that they all share the same first and last letter, say $\hh_a$ and $\hh_b^s$, respectively (assuming that the first letter is a generator simplifies the writing, without too much loss of generality).  Therefore\,:
\begin{align*}
p^+ & \in U_a^+ \\
p^- & \in U_b^{-s} .
\end{align*}

Let $\gg_\infty$ be the infinite word corresponding to the sequence $\{\gg_j\}$\,:
\begin{equation*}
\gg_\infty=\hh_a\hh_{i_1}^{s_1}\hh_{i_2}^{s_2}\ldots\hh_{i_n}^{s_j}\ldots .  
\end{equation*}
Denote the sequence of left-factors of $\gg_\infty$ as follows\,:
\begin{equation*}
\gg_\infty^{(j)}=\hh_a\hh_{i_1}^{s_1}\hh_{i_2}^{s_2}\ldots\hh_{i_j}^{s_j}.  
\end{equation*}
Since $\{\gg_j\}$ is a sequence with bounded distortion, $\{\gg_\infty^{(j)}\}$ is also a sequence with bounded distortion. 

Denote by $K(\gg_\infty)$ the connected component of $\Eint\setminus\Omega$ corresponding to $\gg_\infty$.  Let $U$ be one of the sets $U^\pm_1,\ldots,U^\pm_n$ which appears infinitely often in the sequence of elements~\eqref{eq:nested}.  Without loss of generality, we may assume that $U\neq U_b^{-s}$.  (If need be, take a conjugate of the sequence; this will not affect the dynamics.)  Then $K(\gg_\infty)$ equals the limit of the sequence of sets $\gg_\infty^{(j)}(U)$.

Let $P_U$ be the set of photons in $\econe{p^-}$ which intersect $U$.  This set is non-empty, by Lemma~\ref{lem:nocontain}, since a crooked surface bounds $U$.   By the characterization of a sequence with bounded distortion\,: 
\begin{equation*}
D(P_U)=K(\gg_\infty).
\end{equation*}
What this means is, the only photons belonging to $K(\gg_\infty)$ are those of the form $\mm(\phi)$, where $\phi\in P_U$.

On the other hand, there must exist a photon in $\econe{p^-}$ which does not belong to $P_U$.  Otherwise, every photon in $\econe{p^+}$ would have to be in the complement of $\Omega$ and therefore, $\econe{p^+}\subset U_a^+$, but Lemma~\ref{lem:nocontain} forbids this.

Now we can slightly deform $U$ and its paired open set to obtain a new compact set $U'$ which still belongs to a Schottky configuration for $\Gamma$.  (With some work, we could take $U'$ to be a crooked halfspace, but it is not necessary for the argument.)  We can do this in such a way that\,:
\begin{itemize} 
\item $P_U\cap P_{U'}\neq\emptyset$;
\item there exists a photon $\phi\in P_{U'}$ such that $\phi  \notin P_U$.
\end{itemize}
However, the union of the $\Gamma$-translates of the new sets still equals $\Omega$, since $U'$ is contained in some finite set of $\Gamma$-translates of $F$.  Therefore, $\gg_\infty^{(j)}(U')$ converges to a set, contained in a connected component of $\Eint\setminus\Omega$ that must be different from $K(\gg_\infty)$, because it contains $\mm(\phi)$, yet intersects $K(\gg_\infty)$ since $D(P_U)\cap D(P_{U'})\neq\emptyset$.  This is impossible; therefore $\Gamma$ cannot contain a sequence with bounded distortion.

\end{proof}
\section{Negatively extended crooked surfaces}\label{sec:negative}
 
 This last section is really a short note, which we think merits further exploration.  Recall that the conformal compactification of a negatively extended 
 crooked plane is in the $\SOTT$-orbit of the basic example from~\S\ref{sec:basic}, as are all crooked surfaces, but not in its $\sott$-orbit.  Now in $\E$, a positively extended crooked plane {\em always} intersects a negatively crooked plane.  However, the situation is different in the Einstein Universe.  
 
 Indeed, let $\CS_1$ be the crooked surface of the basic example.  Thus\,: 
 \begin{equation*}
 \CS_1=\conf{\CP{o,(1,0,0)}}.
 \end{equation*}
Now set\,: 
\begin{equation*}
\CS'_1=\conf{\C^-}
\end{equation*}
where $\C^{-}$ is the negatively extended crooked plane with vertex $o$ and director $(1,0,0)$.  Then set\,:
\begin{equation*}
\CS_2=\mu\left(\C^-\right)
\end{equation*}
where $\mu\in\SOTT$ has matrix\,: 
\begin{equation*}
\begin{bmatrix}
-5/6& 1& -12& 12& 5/6 \\
0& -2& 107/12& -109/12& 0 \\
5/9& 5/3& -20& 20& 13/9 \\
  0& -2& 179/12& -181/12& 0 \\
  1/18& 5/3& -20& 20& 35/18
\end{bmatrix}.
\end{equation*}
It involves several calculations, but one can check that $\CS_1$ and $\CS_2$ are disjoint.  



 \bibliographystyle{amsplain}
 \bibliography{Vref}
  
\end{document}